\documentclass[preprint,12pt]{elsarticle}

\usepackage[T1]{fontenc}
\usepackage[utf8]{inputenc}
\usepackage{lmodern}
\usepackage{microtype}
\usepackage{amsmath,amssymb,amsthm,mathtools}
\usepackage[hidelinks]{hyperref}

\biboptions{sort&compress}

\newtheorem{theorem}{Theorem}[section]
\newtheorem{lemma}[theorem]{Lemma}
\newtheorem{proposition}[theorem]{Proposition}
\newtheorem{corollary}[theorem]{Corollary}
\newtheorem{conjecture}[theorem]{Conjecture}
\theoremstyle{remark}
\newtheorem{remark}[theorem]{Remark}
\newtheorem{example}[theorem]{Example}

\DeclareMathOperator{\diam}{diam}
\DeclareMathOperator{\conv}{conv}

\journal{European Journal of Combinatorics}

\hypersetup{
  pdftitle={Counterexamples to the Corsten--Frankl conjecture on diameter-Ramsey simplices},
  pdfauthor={Yaping Mao},
  pdfsubject={Diameter-Ramsey simplices},
  pdfkeywords={diameter-Ramsey set, simplex, circumcenter, Cartesian product, Euclidean Ramsey theory}
}

\begin{document}

\begin{frontmatter}

\title{Counterexamples to the Corsten--Frankl conjecture on diameter-Ramsey simplices}

\author[aff1]{Yaping Mao\corref{cor1}\fnref{fn1}}
\cortext[cor1]{Corresponding author.}
\fntext[fn1]{Supported by the National Science Foundation of China (Nos.\ 12471329 and 12061059).}
\ead{yapingmao@outlook.com; myp@qhnu.edu.cn}

\address[aff1]{Academy of Plateau Science and Sustainability, and School of Mathematics and Statistics, Qinghai Normal University, Xining, Qinghai 810008, China}

\begin{abstract}
Corsten and Frankl conjectured that a simplex is diameter-Ramsey if and only if its circumcenter lies in its convex hull. We disprove this conjecture in every dimension $d\ge 3$. The main tool is a sufficient criterion based on a higher-order deficit decomposition: if the squared deficits $D^2-\|p_i-p_j\|^2$ admit a nonnegative decomposition over subsets of the vertex set, with total mass at most $D^2$, then the simplex is diameter-Ramsey. The pairwise deficit criterion of Frankl--Pach--Reiher--R\"odl is recovered as a special case. As an application, for every $d\ge 3$ we construct a diameter-Ramsey $d$-simplex whose circumcenter lies outside its convex hull. A particularly simple family has squared edge lengths
$\|p_1-p_2\|^2=\|p_1-p_j\|^2=7~ (4\le j\le d+1)$,
$\|p_1-p_3\|^2=4$,
$\|p_i-p_j\|^2=4 ~ (2\le i<j\le d+1)$.
\end{abstract}


\end{frontmatter}

\section{Introduction}

\noindent Euclidean Ramsey theory considers
the problems that
for any positive integer $r$ and any configuration $K$,
does there exist an integer $n_0$ such that
for any $n \geq n_0$ and any $r$-coloring of $\mathbb{E}^n$,
there exists a monochromatic configuration congruent to $K$?
Erd\H{o}s, Graham, Montgomery, Rothschild, Spencer, and
Straus \cite{EGMRSS73} in 1973 first
investigated this problem.
We refer the reader to survey papers \cite{GrahamTressler, GrahamOldNew,
GrahamRecentTrend}.

For a positive integer $m$, write $[m]:=\{1,\dots,m\}$. A finite set $A\subset \mathbb R^m$ is called \emph{diameter-Ramsey} if, for every positive integer $q$, there are an integer $N$ and a finite set $R\subset \mathbb R^N$ such that
\[
\diam(R)=\diam(A)
\]
and every $q$-coloring of $R$ contains a monochromatic subset congruent to $A$. We write $R\to(A)_q$ for this Ramsey property.

Throughout the paper, by a \emph{simplex} we mean its vertex set. Thus an $(n-1)$-simplex is a set
\[
A=\{p_1,\dots,p_n\}
\]
of $n$ affinely independent points in a Euclidean space. Its circumcenter is the center of the unique sphere through its vertices.

Frankl, Pach, Reiher, and R\"odl introduced diameter-Ramsey sets and proved, among other things, that regular simplices are diameter-Ramsey, that Cartesian products of diameter-Ramsey sets are again diameter-Ramsey, and that a pairwise deficit condition suffices for simplices to be diameter-Ramsey \cite{FPRR}. Corsten and Frankl later proved that every finite spherical set with circumradius strictly larger than $\diam(A)/\sqrt{2}$ is not diameter-Ramsey \cite{CF}. Motivated by this, they \cite{CF} proposed the following conjecture.  
\begin{conjecture}
A simplex is diameter-Ramsey if and only if its circumcenter lies in its convex hull.   
\end{conjecture}

The purpose of this paper is to show that this conjecture fails in every dimension $d\ge 3$.

Our main structural observation is a sufficient criterion that may be viewed as a \emph{higher-order deficit decomposition}. Let
\[
A=\{p_1,\dots,p_n\}
\]
be a simplex with diameter $D$, and define its pairwise deficits by
\[
\delta_{ij}:=D^2-\|p_i-p_j\|^2\qquad (1\le i<j\le n).
\]
We show that if the deficits $\delta_{ij}$ admit a suitable nonnegative decomposition over subsets of $[n]$, with total mass at most $D^2$, then $A$ is diameter-Ramsey. The criterion from \cite[Section~4]{FPRR} is recovered when only two-element subsets are used.

\section{Basic lemmas}
We begin with two standard facts and include short proofs for completeness.

\begin{lemma}\label{lem:regular}
Every regular simplex is diameter-Ramsey.
\end{lemma}

\begin{proof}
Let $\Delta_k$ be a regular $k$-simplex of side length $\ell$, and fix a positive integer $q$. Consider a regular simplex $V$ with $qk+1$ vertices and side length $\ell$. Then $\diam(V)=\ell=\diam(\Delta_k)$. In every $q$-coloring of $V$, some color appears on at least $k+1$ vertices. Any $k+1$ vertices of a regular simplex form a regular $k$-simplex of side length $\ell$, so $V$ contains a monochromatic copy of $\Delta_k$. Hence $\Delta_k$ is diameter-Ramsey.
\end{proof}

\begin{lemma}\label{lem:product}
If $X\subset\mathbb R^m$ and $Y\subset\mathbb R^n$ are diameter-Ramsey, then their Cartesian product
\[
X\times Y:=\{(x,y):x\in X,\ y\in Y\}\subset\mathbb R^{m+n},
\]
with the Euclidean product metric, is diameter-Ramsey. Moreover,
\[
\diam(X\times Y)^2=\diam(X)^2+\diam(Y)^2.
\]
\end{lemma}

\begin{proof}
Fix a positive integer $q$. Since $X$ is diameter-Ramsey, there is a finite set $S$ such that
\[
\diam(S)=\diam(X)
\qquad\text{and}\qquad
S\to(X)_q.
\]
Let $\mathcal P$ be the finite set of all pairs $(X',c)$ where $X'\subseteq S$, $X'\cong X$, and $c\in[q]$. Put $M:=|\mathcal P|$.

Since $Y$ is diameter-Ramsey, there is a finite set $T$ such that
\[
\diam(T)=\diam(Y)
\qquad\text{and}\qquad
T\to(Y)_M.
\]
We claim that $S\times T$ witnesses that $X\times Y$ is diameter-Ramsey for $q$ colors.

Let $\chi:S\times T\to[q]$ be a $q$-coloring. For each $y\in T$, the coloring
\[
x\longmapsto \chi(x,y)\qquad (x\in S)
\]
contains a monochromatic copy $X_y\subseteq S$ congruent to $X$. Choose one such copy and denote its color by $c_y$. Thus each $y\in T$ receives the auxiliary color
\[
\phi(y):=(X_y,c_y)\in\mathcal P.
\]
Since $T\to(Y)_M$, there exists a subset $Y'\subseteq T$ congruent to $Y$ on which $\phi$ is constant. Hence there are a fixed copy $X_*\subseteq S$ congruent to $X$ and a fixed color $c_*$ such that $X_y=X_*$ and $c_y=c_*$ for all $y\in Y'$. Therefore every point of $X_*\times Y'$ has color $c_*$, and $X_*\times Y'$ is a monochromatic copy of $X\times Y$.

The formula for the diameter follows from the product metric:
\[
\diam(S\times T)^2=\diam(S)^2+\diam(T)^2
=\diam(X)^2+\diam(Y)^2.
\]
Therefore $X\times Y$ is diameter-Ramsey.
\end{proof}

By induction, Lemma~\ref{lem:product} applies to every finite Cartesian product of diameter-Ramsey sets.

\section{A higher-order deficit criterion}

\begin{theorem}\label{thm:criterion}
Let $A=\{p_1,\dots,p_n\}$ be an $(n-1)$-simplex with $n\ge 2$, and let
\[
D:=\diam(A).
\]
Relabel the vertices so that $\|p_1-p_2\|=D$. Let
\[
\mathcal B:=\{B\subseteq[n]: |B|\ge 2\ \text{and}\ \{1,2\}\nsubseteq B\}.
\]
Suppose that there are nonnegative real numbers $\alpha_B$, indexed by $B\in\mathcal B$, such that
\begin{equation}\label{eq:deficits}
D^2-\|p_i-p_j\|^2
=
\sum_{\substack{B\in\mathcal B\\ \{i,j\}\subseteq B}}\alpha_B
\qquad (1\le i<j\le n),
\end{equation}
and
\begin{equation}\label{eq:mass}
\sum_{B\in\mathcal B}\alpha_B\le D^2.
\end{equation}
Then $A$ is diameter-Ramsey.
\end{theorem}

\begin{proof}
Define the reserve mass
\[
\alpha_0:=D^2-\sum_{B\in\mathcal B}\alpha_B\ge 0.
\]
We shall realize $A$ inside a finite Cartesian product of regular simplices whose diameter is $D$.

If $\alpha_0>0$, let $S_0$ be a regular $(n-1)$-simplex of side length $\sqrt{\alpha_0}$ with distinct vertices
\[
v_1^{(0)},\dots,v_n^{(0)}.
\]
If $\alpha_0=0$, let $S_0$ be a singleton and put $v_1^{(0)}=\cdots=v_n^{(0)}$ equal to its unique point.

For each $B\in\mathcal B$ with $\alpha_B>0$, let $S_B$ be a regular simplex of side length $\sqrt{\alpha_B}$ with vertex set
\[
\{u_B\}\cup\{v_i^{(B)}:i\notin B\}.
\]
This is possible because the displayed set has $n-|B|+1$ vertices. Define
\[
q_i^{(B)}:=
\begin{cases}
 u_B, & i\in B,\\[1mm]
 v_i^{(B)}, & i\notin B,
\end{cases}
\qquad (i\in[n]),
\]
and set $q_i^{(0)}:=v_i^{(0)}$.

Now put
\[
\mathcal R:=S_0\times \prod_{\substack{B\in\mathcal B\\ \alpha_B>0}} S_B
\]
and
\[
q_i:=\bigl(q_i^{(0)},(q_i^{(B)})_B\bigr)\in\mathcal R
\qquad (i\in[n]).
\]

Fix $1\le i<j\le n$. In the factor $S_B$, the points $q_i^{(B)}$ and $q_j^{(B)}$ coincide if and only if $\{i,j\}\subseteq B$; otherwise they are distinct vertices of a regular simplex and are at squared distance $\alpha_B$. Therefore
\begin{align*}
\|q_i-q_j\|^2
&=\alpha_0+
\sum_{\substack{B\in\mathcal B\\ \alpha_B>0\\ \{i,j\}\nsubseteq B}}\alpha_B \\
&=D^2-
\sum_{\substack{B\in\mathcal B\\ \{i,j\}\subseteq B}}\alpha_B \\
&=\|p_i-p_j\|^2,
\end{align*}
where the last equality is \eqref{eq:deficits}. Thus $\{q_1,\dots,q_n\}$ is congruent to $A$.

Moreover, no set $B\in\mathcal B$ contains both $1$ and $2$. Hence $q_1$ and $q_2$ are distinct in every nontrivial factor, and
\[
\|q_1-q_2\|^2
=\alpha_0+
\sum_{B\in\mathcal B}\alpha_B
=D^2.
\]
The same sum is the squared diameter of the product $\mathcal R$, so
\[
\diam(\mathcal R)=D.
\]

By Lemma~\ref{lem:regular}, every nontrivial factor of $\mathcal R$ is diameter-Ramsey, and a singleton is trivially diameter-Ramsey. Repeated use of Lemma~\ref{lem:product} shows that $\mathcal R$ is diameter-Ramsey. Therefore, for every positive integer $q$, there exists a finite set $W_q$ such that
\[
\diam(W_q)=\diam(\mathcal R)=D
\qquad\text{and}\qquad
W_q\to(\mathcal R)_q.
\]
Since $\mathcal R$ contains the subset $\{q_1,\dots,q_n\}$ congruent to $A$, every monochromatic copy of $\mathcal R$ contains a monochromatic copy of $A$. Hence $W_q\to(A)_q$ for every $q$, and $A$ is diameter-Ramsey.
\end{proof}

\begin{remark}\label{rem:pairwise}
If only two-element subsets are used in Theorem~\ref{thm:criterion}, namely if
\[
\alpha_{\{i,j\}}=D^2-\|p_i-p_j\|^2
\qquad ((i,j)\ne(1,2)),
\]
and all other coefficients are zero, then \eqref{eq:mass} becomes
\[
\sum_{1\le i<j\le n}\bigl(D^2-\|p_i-p_j\|^2\bigr)\le D^2,
\]
because the $(1,2)$-term is zero. This is exactly the sufficient condition proved by Frankl--Pach--Reiher--R\"odl in \cite[Section~4]{FPRR}.
\end{remark}

\section{A counterexample family}
We now specialize the construction in the proof of Theorem~\ref{thm:criterion} to a simple three-parameter family. The first parameter is denoted by $s$ to avoid confusing it with the number of colors in the Ramsey property.

\begin{proposition}\label{prop:familyDR}
Let $d\ge 3$ and let $s,t,u>0$. Then there exists a diameter-Ramsey $d$-simplex
\[
A_d(s,t,u)=\{p_1,\dots,p_{d+1}\}
\]
with squared edge lengths
\begin{align}
\|p_1-p_2\|^2&=\|p_1-p_j\|^2=s+t+u && (4\le j\le d+1), \label{eq:family1}\\
\|p_1-p_3\|^2&=s+u, \label{eq:family2}\\
\|p_i-p_j\|^2&=s+t && (2\le i<j\le d+1). \label{eq:family3}
\end{align}
\end{proposition}

\begin{proof}
Run the construction from the proof of Theorem~\ref{thm:criterion} with $n=d+1$ and
\[
D^2:=s+t+u.
\]
Choose the only nonzero higher-order coefficients to be
\[
\alpha_{\{1,3\}}=t,
\qquad
\alpha_{\{2,3,\dots,d+1\}}=u.
\]
The total non-reserve mass is $t+u\le s+t+u=D^2$, so the reserve mass is $\alpha_0=s$.

The construction gives points $q_1,\dots,q_{d+1}$ in a Cartesian product of regular simplices such that
\[
D^2-\|q_1-q_3\|^2=t,
\qquad
D^2-\|q_i-q_j\|^2=u
\quad (2\le i<j\le d+1),
\]
and all remaining deficits are zero. Equivalently, the squared distances among the $q_i$ are precisely those listed in \eqref{eq:family1}--\eqref{eq:family3}. The product $\mathcal R$ used in this construction has diameter $D$ and is diameter-Ramsey; since it contains $\{q_1,\dots,q_{d+1}\}$, the final argument in the proof of Theorem~\ref{thm:criterion} shows that this set is diameter-Ramsey.

It remains only to check that these points form a $d$-simplex. Since $\alpha_0=s>0$, the projection onto the $S_0$-factor maps $q_1,\dots,q_{d+1}$ to the vertices of a regular $d$-simplex. Any affine dependence among $q_1,\dots,q_{d+1}$ would project to an affine dependence among those $d+1$ regular-simplex vertices. Hence no nontrivial affine dependence exists, and the points form a $d$-simplex. We denote this simplex by $A_d(s,t,u)$.
\end{proof}

Next we locate the circumcenter.

\begin{proposition}\label{prop:bary}
Let $d\ge 3$, and let $A_d(s,t,u)$ be the simplex from Proposition~\ref{prop:familyDR}. Write the circumcenter in barycentric coordinates as
\[
O=\lambda_1p_1+\cdots+\lambda_{d+1}p_{d+1}.
\]
Then
\[
\lambda_2=\lambda_4=\cdots=\lambda_{d+1},
\]
and, more precisely,
\begin{align*}
\lambda_1&=\frac{(s+t)(s+du)}{\Delta_d},\\[1mm]
\lambda_2=\lambda_4=\cdots=\lambda_{d+1}
&=\frac{(s+2t)(s+u)}{\Delta_d},\\[1mm]
\lambda_3&=\frac{s^2+st+su-(d-2)tu}{\Delta_d},
\end{align*}
where
\[
\Delta_d:=(d+1)s^2+2d(st+su+tu).
\]
Consequently, the circumcenter lies outside $\conv(A_d(s,t,u))$ whenever
\begin{equation}\label{eq:outsidecondition}
(d-2)tu>s(s+t+u).
\end{equation}
\end{proposition}

\begin{proof}
The distance data in \eqref{eq:family1}--\eqref{eq:family3} are invariant under every permutation of the indices $2,4,\dots,d+1$. The labelled configurations before and after such a permutation have the same distance matrix; by rigidity of Euclidean simplices, the permutation is induced by an isometry of the affine span. This isometry fixes the circumcenter, and the barycentric representation in a simplex is unique. Hence the barycentric coordinates corresponding to the indices $2,4,\dots,d+1$ are equal. Write
\[
\lambda_1=a,
\qquad
\lambda_3=c,
\qquad
\lambda_2=\lambda_4=\cdots=\lambda_{d+1}=b.
\]

Let $\rho$ be the circumradius and translate the circumcenter to the origin. Then $\|p_i\|=\rho$ for every $i$, and
\[
\|p_i-p_j\|^2=2\rho^2-2\langle p_i,p_j\rangle.
\]
Since $\sum_j\lambda_jp_j=0$ and $\sum_j\lambda_j=1$, we obtain
\begin{equation}\label{eq:circsys}
\sum_{j=1}^{d+1}\lambda_j\|p_i-p_j\|^2=2\rho^2
\qquad (1\le i\le d+1).
\end{equation}

Using \eqref{eq:family1}--\eqref{eq:family3}, the cases $i=1$, $i=3$, and a representative case $i=2$ among the indices $2,4,\dots,d+1$ of \eqref{eq:circsys} give
\begin{align}
(s+u)c+(d-1)(s+t+u)b&=2\rho^2, \label{eq:bary1}\\
(s+u)a+(d-1)(s+t)b&=2\rho^2, \label{eq:bary2}\\
(s+t+u)a+(s+t)c+(d-2)(s+t)b&=2\rho^2. \label{eq:bary3}
\end{align}
The normalization condition is
\begin{equation}\label{eq:bary4}
a+c+(d-1)b=1.
\end{equation}
Subtracting \eqref{eq:bary1} from \eqref{eq:bary2} and subtracting \eqref{eq:bary2} from \eqref{eq:bary3} reduce the system to
\[
(s+u)(a-c)=(d-1)ub,
\qquad
(s+t)(b-c)=ta,
\qquad
 a+c+(d-1)b=1.
\]
Solving these three linear equations gives
\begin{align*}
a&=\frac{(s+t)(s+du)}{\Delta_d},\\[1mm]
b&=\frac{(s+2t)(s+u)}{\Delta_d},\\[1mm]
c&=\frac{s^2+st+su-(d-2)tu}{\Delta_d},
\end{align*}
with
\[
\Delta_d=(d+1)s^2+2d(st+su+tu).
\]
This proves the displayed formulae.

Because $s,t,u>0$, the denominator $\Delta_d$ is positive, and the factorizations of the numerators of $\lambda_1$ and $\lambda_2$ show that they are positive. Thus the only barycentric coordinate that can be negative is $\lambda_3$. A point belongs to the convex hull of a simplex if and only if all of its barycentric coordinates are nonnegative. Therefore the circumcenter lies outside $\conv(A_d(s,t,u))$ exactly when $\lambda_3<0$, which is equivalent to \eqref{eq:outsidecondition}.
\end{proof}

Combining the preceding two propositions yields the desired counterexamples.

\begin{theorem}\label{thm:main}
For every integer $d\ge 3$, there exists a diameter-Ramsey $d$-simplex whose circumcenter lies outside its convex hull.
\end{theorem}

\begin{proof}
Take $s=1$ and $t=u=3$. By Proposition~\ref{prop:familyDR}, there exists a diameter-Ramsey $d$-simplex with squared edge lengths
\[
\begin{aligned}
\|p_1-p_2\|^2&=\|p_1-p_j\|^2=7 && (4\le j\le d+1),\\
\|p_1-p_3\|^2&=4,\\
\|p_i-p_j\|^2&=4 && (2\le i<j\le d+1).
\end{aligned}
\]
For this simplex, Proposition~\ref{prop:bary} gives
\[
\lambda_3=
\frac{1+3+3-(d-2)\cdot 9}{(d+1)+2d(3+3+9)}
=\frac{25-9d}{31d+1}<0
\qquad (d\ge 3).
\]
Thus the circumcenter lies outside the convex hull of the simplex.
\end{proof}

\begin{corollary}\label{cor:strict}
For every integer $d\ge 3$, there exists a diameter-Ramsey $d$-simplex that is not covered by the pairwise deficit criterion from \cite[Section~4]{FPRR}.
\end{corollary}

\begin{proof}
For the simplex in Theorem~\ref{thm:main}, we have $D^2=7$,
\[
D^2-\|p_1-p_3\|^2=3,
\qquad
D^2-\|p_i-p_j\|^2=3
\quad (2\le i<j\le d+1),
\]
and all other pairwise deficits are zero. Therefore
\[
\sum_{1\le i<j\le d+1}\bigl(D^2-\|p_i-p_j\|^2\bigr)
=3+3\binom d2
>7=D^2
\qquad (d\ge 3).
\]
Hence the pairwise criterion stated in Remark~\ref{rem:pairwise} does not apply.
\end{proof}

\begin{example}[A concrete tetrahedron]\label{ex:tetra}
For $d=3$, the above construction gives a tetrahedron with squared edge lengths
\[
\begin{aligned}
\|p_1-p_2\|^2&=\|p_1-p_4\|^2=7,\\
\|p_1-p_3\|^2&=\|p_2-p_3\|^2=\|p_2-p_4\|^2=\|p_3-p_4\|^2=4.
\end{aligned}
\]
Its circumcenter has barycentric coordinates
\[
\left(\frac{20}{47},\frac{14}{47},-\frac{1}{47},\frac{14}{47}\right),
\]
so it lies outside the tetrahedron.

One explicit realization is
\[
p_2=(0,0,0),
\qquad
p_3=(2,0,0),
\qquad
p_4=(1,\sqrt{3},0),
\]
\[
p_1=\left(\frac{7}{4},\frac{\sqrt{3}}{12},\sqrt{\frac{47}{12}}\right).
\]
The circumcenter is
\[
O=\left(1,\frac{\sqrt{3}}{3},\frac{10\sqrt{141}}{141}\right).
\]
Its barycentric coordinate at $p_3$ is negative, so $O$ lies across the face
\[
\conv\{p_1,p_2,p_4\}.
\]
\end{example}

\section{Final remarks}
The Corsten--Frankl conjecture therefore fails in every dimension $d\ge 3$. In particular, the position of the circumcenter alone does not characterize diameter-Ramsey simplices. Theorem~\ref{thm:criterion} suggests that higher-order deficit decompositions may be a useful tool for studying diameter-Ramsey simplices.

The argument here does not address the planar case $d=2$.

\end{document}